\documentclass[11pt]{amsart}
\usepackage{amsmath,amssymb,amsfonts,amsthm,amscd,indentfirst}

\numberwithin{equation}{section}

\newtheorem{prop}{Proposition}[section]
\newtheorem{theo}[prop]{Theorem}
\newtheorem{lemm}[prop]{Lemma}
\newtheorem{coro}[prop]{Corollary}
\newtheorem{rema}[prop]{Remark}

\def\begeq{\begin{equation}}
\def\endeq{\end{equation}}

\def\and{\quad{\rm and}\quad}

\def\<{\langle}
\def\>{\rangle}

\def\P{\partial}
\def\O{\Omega}

\def\di{\displaystyle}
\def\Dint{\displaystyle\int}

\begin{document}
\title[Curvature flows and inequalities]
{A mean curvature type flow in space forms}
\author{Pengfei Guan and Junfang Li}
\address{Department of Mathematics\\
         McGill University\\
         Montreal, Quebec. H3A 2K6, Canada.}
\email{guan@math.mcgill.ca}

\address{Department of Mathematics\\
         University of Alabama at Birmingham\\
         Birmingham, AL 35294}
	 \email{jfli@uab.edu}
\date{}
\thanks{Research of the first author was supported in part by NSERC Discovery Grant. Research of the second author was supported in part by NSF DMS-1007223.}

\begin{abstract}
  In this article, we introduce a new type of mean curvature flow (\ref{umcf general}) for bounded star-shaped domains in space forms and prove its longtime existence, exponential convergence without any curvature assumption. Along this flow, the enclosed volume is a constant and the surface area evolves monotonically. Moreover, for a bounded convex domain in $\mathbb R^{n+1}$, the quermassintegrals evolve monotonically along the flow which allows us to prove a class of Alexandrov-Fenchel inequalities of quermassintegrals. 
\end{abstract}
\maketitle

\section{Introduction}
There have been extensive interests on flows of hypersurfaces evolving by functions of their mean curvature in the past $30$ years. Brakke \cite{brakke} used geometric measure theory to study surfaces driven by their mean curvature. The more differential geometric approach was given by Huisken \cite{H} who studied the mean curvature flow,
\begin{equation}\label{mcf}
    \frac{\partial X}{\partial t}=-H\nu,
\end{equation}
where $X$, $\nu$  and $H$ are the position function, the outward unit normal vector and the mean curvature of the hypersurface respectively.

Huisken in \cite{H} proved that flow (\ref{mcf}) is a contracting flow which contracts convex hypersurfaces into a round point. In contrast, another model type of flow, the inverse mean curvature flow, is an expanding flow introduced by Gerhardt \cite{G}, Urbas \cite{U}, which expands star-shaped mean convex hypersurfaces out to a round sphere.  Note that, in general, flow (\ref{mcf}) would develop singularities when the initial data is a star-shaped mean convex hypersurface.

 Hypersurface flows can be used to prove geometric inequalities, old and new. For example, Gage and Hamilton \cite{gage} used curve shortening flow to prove isoperimetric inequality for convex planar domains. Andrews \cite{Andrews0} used affine mean curvature flow to prove affine isoperimetric inequality.   Topping \cite{topping} studied generalized isoperimetric inequalities for convex domains on Riemann surfaces using Grayson's curve shortening flow. Various geometric inequalities were proved in \cite{GW1, GW2} using fully nonlinear conformal flows. Schulze \cite{Sch} used a non-linear mean curvature type flow to prove the isoperimetric inequality in $\mathbb R^{n+1}$ with $n\le 7$. The authors used the inverse mean curvature type of flows to prove Alexandrov-Fenchel quermassintegral inequalities for starshaped $k$-convex domains in \cite{GL}.

In this paper, we introduce a new type of mean curvature flow in space forms. For example, in Euclidean space $\mathbb R^{n+1}$ it reads as
\begin{equation}
    \frac{\partial X}{\partial t}=(n-Hu)\nu.
    \label{umcf}
  \end{equation}
where $X$ is the position vector, $H$ is the mean curvature, and $u= <X, \nu>$ the support function of the hypersurface.
In general space forms, the flow equation is given by
\begin{equation}
  \P_t X=(n\phi'(\rho)-Hu)\nu,
  \label{umcf general}
\end{equation}
while more details of the definition of this flow in general space forms can be found in section \ref{sec flow}, c.f. (\ref{uMCFh}). We will show this flow evolves star-shaped domains ({\it without any curvature assumption}) in space forms into spheres. The key of discovering this flow is the Minkowski identity which also plays crucial role in the proof. We note that the assumption of starshapedness (i.e., $u>0$) is necessary for flow (\ref{umcf general}) to be parabolic.

One of the main theorems of this article is the following.
\begin{theo}\label{main thm}
Let $M_0$ be a smooth compact, star-shaped hypersurface with respect to the origin in the space forms $N^{n+1}(K)$ with $n\ge2$. Then the evolution equation (\ref{umcf general}) has a smooth solution for $t\in [0,\infty)$. Moreover, the surfaces converges exponentially to a sphere as $t\rightarrow \infty$ in the $C^\infty$ topoloty.
\end{theo}

Another nice feature of flow (\ref{umcf general}) is that along the flow, volume of the enclosed domain is constant and the surface area is monotonically decreasing. This property yields that the flow of the hypersurface $M_t$ converges to a solution of the isoperimetric problem in the space forms.

If the initial hypersurface encloses a (weakly) convex bounded domain in Euclidean space $\mathbb R^{n+1}$, then by arguments in Bian-Guan \cite{BG}, the convexity is preserved along the flow (\ref{umcf}) and the domain becomes strictly convex instantly. Moreover, we will show all the quermassintegrals
\[
\Dint_{M}\sigma_k(\kappa)d\mu_g
\]
are non-increasing along the flow, where the quermassintegral is surface area, total mean curvature, total scalar curvature, etc. repectively for different $k=0,1,2,\cdots, n-1$. This yields a proof to the following well-known Alexandrov-Fenchel inequalities of quermassintegrals from convex geometry.

\begin{coro}\label{cor 1}
   Suppose $\Omega\subset \mathbb R^{n+1}$ is a bounded convex domain with smooth boundary $\partial \Omega$. Then the following inequality holds for any $0\le k<n-1$,
  \begin{equation}
    V(\Omega)^{\frac{1}{n+1}}\le c_{n,k}\Big(\Dint_{\P\Omega}\sigma_k(\kappa)d\mu\Big)^\frac{1}{n-k},
    \label{}
  \end{equation}
  where $c_{n,k}$ is a constant depending only on $n$ and $k$, and $\sigma_k(\kappa)$ is the $k$th elementary symmetric function of the principal curvatures. The ``$=$'' is attained if and only if $\O$ is a ball.\\
\end{coro}

When $k=0$, we have used the convention $\sigma_0\equiv 1$. Note that in this case, the inequality is reduced to the classical isoperimetric inequality. In fact, in this case the convexity condition can be replaced by star-shapedness. \\

\begin{rema}
  In our second paper \cite{GL1}, we will present a fully-nonlinear generalization of the mean curvature type of flow which can be used to prove other types of quermassintegral inequalities in Euclidean space.\\
\end{rema}

In appearance, flow (\ref{umcf}) resembles to the normalized mean curvature flow of Huisken \cite{H},
\begin{equation}
    \frac{\partial X}{\partial t}=(c(t)u-H)\nu,
	\label{nmcf}
\end{equation}
where $c(t)=\frac{1}{n}\frac{\int_MH^2d\mu}{\int_Md\mu}$ for mean curvature $H$ and also the volume preserving mean curvature flow of Huisken \cite{H4},

\begin{equation}
    \frac{\partial X}{\partial t}=(c(t)-H)\nu,
	\label{vmcf}
\end{equation}
where $c(t)=\frac{\int_MHd\mu}{\int_Md\mu}$.
However, there are several differences among these equations. The longtime existence of solutions of (\ref{nmcf}) or (\ref{vmcf}) is not clear if the initial surface is not convex, while (\ref{umcf}) requires no curvature assumption except starshapedness. Note that normalized mean curvautre flow (\ref{nmcf}) preserves the surface area while the volume preserving mean curvature flow (\ref{vmcf}) and our new flow (\ref{umcf}) preserves the volume and evolve the surface area monotonically. Moreover, (\ref{umcf}) evovles all the quermassintegrals monotonically.

On the other hand, the presence of normalization factor $c(t)$ makes (\ref{nmcf}) or (\ref{vmcf}) an integral-differential equation and a priori estimates are not easy to obtain. Thus, for the quasilinear equation (\ref{nmcf}) or (\ref{vmcf}), one needs to do the more involved curvature estimates rather than just gradient estimate. In contrast, for flow (\ref{umcf}), the normalization function is a fixed constant $n$ thanks to the Minkowski identity, and the principal part of the resulting parabolic equation is of divergent form, which makes the a priori estimates easier by the classical theory of parabolic equations in divergent form (e.g., \cite{Lady}).

We note there is also a volume preserving mean curvature flow defined in hyperbolic space, see \cite{CM}. \\

When $n=1$, Theorem \ref{main thm} holds for curves as well. We will leave the details for the readers. The rest of this paper is organized as follows. In Section \ref{sec hypersurfaces}, we give the preliminaries for hypersurface theory in space forms and prove the important Minkowski type identities. In Section \ref{sec flow}, we introduce the flow equation in space forms, derive the evolution equations for various geometric quantities, and prove the monotonicity formulas. In Section \ref{sec C0}, we convert the problem to an initial value problem on $\mathbb S^n$ and prove the $C^0$ estimate. In Section \ref{sec estimates}, we establish the gradient estimate and prove the exponential convergence and the main theorems. In the last section, we show the convexity is preserved along the flow in Euclidean space and prove Corollary \ref{cor 1}.

\section{Hypersurfaces in space forms}\label{sec hypersurfaces}

Let $N^{n+1}(K)$ be a space form of sectional curvature $K=-1, 0,$ or $+1$ and $n\ge 2$. Let $g^N:=ds^2$ denote the Riemannian metric of $N^{n+1}(K)$. Then it is known that the space forms can be viewed as Euclidean space $\mathbb R^{n+1}$ equipped with a metric tensor, i.e., $N^{n+1}(K)=(\mathbb R^{n+1},ds^2)$ with proper choice $ds^2$. More specifically, let $\mathbb S^n$ be the unit sphere in Euclidean space $\mathbb R^{n+1}$ with standard induced metric $dz^2$, then

\begin{equation}
  g^N:= ds^2=d\rho^2+\phi^2(\rho)dz^2,
  \label{}
\end{equation}
where $N^{n+1}$ is the Euclidean space $\mathbb R^{n+1}$ if $\phi(\rho)=\rho$, $\rho\in [0,\infty)$; $N^{n+1}$ is the elliptic space $\mathbb S^{n+1}$ with constant sectional curvature $+1$ if $\phi(\rho)=\sin(\rho)$, $\rho\in [0,\pi)$; and $N^{n+1}$ is the hyperbolic space $\mathbb H^{n+1}$ with constant sectional curvature $-1$ if $\phi(\rho)=\sinh(\rho)$, $\rho\in [0,\infty)$. \\

We define
\begin{equation}
  \Phi(\rho)=\Dint^\rho_0\phi(r)dr.
  \label{Phi}
\end{equation}
Then $\Phi(\rho)$ is $\frac{\rho^2}{2}$, $\cosh\rho$, $-\cos\rho$ respectively. Consider the vector field $V=\phi(\rho)\frac{\P}{\P \rho}$ on $N^{n+1}(K)$. We first show that $V$ is a conformal killing field. The following lemma holds for general warped product manifolds, see for example, \cite{P}. For convenience, we include a proof here.

\begin{lemm}
  \label{lemm conformal}
  The vector field $V$ satisfies $D_iV_j=\phi'(\rho)g^N_{ij}$, where $D$ is the covariant derivative with respect to the metric $g^N$.
\end{lemm}

\begin{proof}
 We first note that the Lie derivative of $d\rho$ is
 \[
 \mathcal L_{V} d\rho= \mathcal L_{\phi(\rho)\frac{\P}{\P\rho}} d\rho = \phi'(\rho)d\rho.
 \]
Thus
\[
\mathcal L_{V}d\rho^2= 2\phi'(\rho)d\rho^2.
\]

Note
\[
\mathcal L_{V} \phi^2(\rho)dz^2= V(\phi^2)dz^2 = 2\phi^2(\rho)\phi'(\rho)dz^2.
\]
We have
\[
\mathcal L_{V} ds^2= 2\phi'(\rho) ds^2.
\]

The proof is now complete since $\frac{1}{2}\mathcal L_{V} g^N = DV$.

\end{proof}

Let $M^n\subset N^{n+1}$ be a closed hypersurface with induced metric $g$.
\begin{lemm}
  \label{lemm hessian}
  Let $M^n\subset N^{n+1}$ be a closed hypersurface with induced metric $g$. Let $\Phi$ and $V$ be defined as in (\ref{Phi}). Then $\Phi|_{M}$ satisfies,
  \begin{equation}
    \nabla_{i}\nabla_j\Phi = \phi'(\rho) g_{ij}-h_{ij}\<V,\nu\>,
    \label{lemm hessian Phi}
  \end{equation}
  where $\nabla$ is the covariant derivative with respect to $g$, $\nu$ is the outward unit normal and $h_{ij}$ is the second fundamental form of the hypersurface.  Moreover, the Laplacian of $\Phi$ satisfies,
  \begin{equation}
    \Delta\Phi = n\phi'(\rho) -H\<V,\nu\>,
    \label{lemm mean}
  \end{equation}
  where $H$ is the mean curvature.
\end{lemm}

\begin{proof}
  Let $e_1,\cdots, e_n$ be a basis of the tangent space of $M$.
  \begin{equation}
    \begin{array}[]{rll}
    \nabla_{i}\nabla_j\Phi =& D^2\Phi(e_i,e_j)-h_{ij}\<V,\nu\>\\
    =&\<D_{e_i}V,e_j\> -h_{ij}\<V,\nu\>\\
    =& \phi'g_{ij}-h_{ij}\<V,\nu\>,
    \end{array}
    \label{}
  \end{equation}
  where the last step we have used Lemma \ref{lemm conformal}.
\end{proof}

Note that in Euclidean space the vector field $V = \rho\frac{\P}{\P\rho}$ is equivalent to the position vector field $X$. Thus $\<V,\nu\>= \<X,\nu\>$ is just the Euclidean support function of the hypersurface $M$. In this sense, let us still call the inner product
\[
u:=\<V, \nu\>
\]
to be the support function of a hypersurface in $N^{n+1}$. An immediate corollary of Lemma \ref{lemm hessian} is the following Minkowski identity in $N^{n+1}$.

\begin{prop}\label{prop Minkowski}
  Let $M$ be a closed hypersurface in $N$. Then
  \begin{equation}
  \Dint_{M}Hu=n\Dint_{M}\phi'(\rho).
    \label{Minkowski}
  \end{equation}
\end{prop}

\begin{proof}
  Applying divergence theorem to (\ref{lemm mean}), we obtain
  \[
  0=\Dint_{M}\Delta \Phi= n\Dint_{M}\phi'(\rho)-\Dint_{M}Hu.
  \]
  This finishes the proof.
\end{proof}

The eigenvalues of the Weingarten tensor $h^i_{j}:= g^{ik}h_{kj}$ are the principal curvatures $\kappa=(\kappa_1,\cdots,\kappa_n)$. If we denote the $k$th elementary symmetric functions of an $n$-vector by $\sigma_k$, then $\sigma_1(\kappa)$ is the mean curvature (un-normalized), $\sigma_2(\kappa)$ is the scalar curvature, etc. For convenience, we also use $\sigma_k(A_{ij})$ to denote the $k$-th elementary function of the eigenvalues of a symmetric matrix $A_{ij}$. Recall in \cite{R1,R2}, Reilly proved a formula for the $r$-th Newton operator $T_r(h)$ of the second fundamental form in general Riemannanian manifolds, see Proposition 1 in \cite{R2}. In space form, since second fundamental form is Codazzi, Reilly's formula is equivalent to say the Newton operator $T_r$ is divergent free. More specifically, in local coordinates, the Newton operator $T_r$ of $h^i_j$ is a symmetric matrix $(T_r)_{ij}=\frac{\P\sigma_{r+1}}{\P h^i_j}=:\sigma_{r+1}^{ij}$, and
\[
\nabla_i\left( T_r \right)_{ij}=0.
\]

Using this property, we can generalize (\ref{lemm mean}) to the $k$th Weingarten curvature $\sigma_k(\kappa)$.

\begin{lemm}
  Let $M^n\subset N^{n+1}$ be a closed hypersurface with induced metric $g$. Let $\Phi$ and $V$ be defined as in (\ref{Phi}). Then $\Phi|_{M}$ satisfies,
  \begin{equation}
    \nabla_{i}(\left( T_k \right)_{ij}\nabla_j\Phi) = (n-k)\phi'\sigma_k(\kappa) -(k+1)\sigma_{k+1}(\kappa)\<V,\nu\>.
    \label{k-Laplacian}
  \end{equation}

\end{lemm}
\begin{proof}
  We multiply (\ref{lemm hessian Phi} ) with $(T_k)_{ij}$, then sum over $i,j$. By divergent free property of $T_{k}$, we have the left hand side is $\nabla_{i}(\left( T_k \right)_{ij}\nabla_j\Phi)$. The following formulas are well-known, see e.g. \cite{H3},
  \[
  \begin{array}[]{rll}
   (T_k)_{ij} g_{ij}=&(n-k)\sigma_k\\
    (T_k)_{ij} h_{ij}= &(k+1) \sigma_{k+1}.
  \end{array}
  \]
  Put these identities together, we complete the proof.
\end{proof}

Prove as in Proposition \ref{prop Minkowski}, we have the following generalized Minkowski identities.

\begin{prop}\label{prop general Minkowski}
  Let $M$ be a closed hypersurface in $N$. Then, for $k=0,1,\cdots, n-1$,
  \begin{equation}
 (k+1)   \Dint_{M}\sigma_{k+1}(\kappa)u=(n-k)\Dint_{M}\phi'(\rho)\sigma_k(\kappa),
    \label{general Minkowski}
  \end{equation}
  where we use the convention that $\sigma_0\equiv 1$.
\end{prop}

Next, we derive the gradient and hessian of the support function $u$ under the induced metric $g$ on $M$. 

\begin{lemm}\label{hessian of u : lemm}
The support function $u$ satisfies
  \begin{equation}
    \begin{array}[]{rll}
      \nabla_i u =& g^{kl}h_{ik}\nabla_l\Phi,\\
      \nabla_i\nabla_j u = &g^{kl} \nabla_{k}h_{ij}\nabla_l \Phi+ \phi'h_{ij}-(h^2)_{ij}u,
    \end{array}
    \label{hessian of u}
  \end{equation}
  where $(h^2)_{ij}:=g^{kl}h_{ik}h_{jl}$.
\end{lemm}

\begin{proof}
Using another model of the space forms, these formulas were first proved in \cite{BLO}. The same formulas in hyperbolic space were also proved by a direct computation in \cite{JL} using the expression of support function $u$ for a radial graph in hyperbolic space, see the expression of $u$ for a radial graph in (\ref{tensors:rho}).

  By the tensorial properties of both sides, we only need to prove them at any point under orthonormal coordinates. Thus, we have $g_{ij}=\delta_{ij}$ and $\nabla_i u= D_i\<V, \nu\>=\<D\Phi, D_i\nu\>=h_{ik}D_k\Phi$. By the tensorial property of both sides, we proved the first identity.

  Differentiate $\nabla u$ one more time and apply (\ref{lemm hessian Phi}), we obtain
\[
\begin{array}[]{rll}
  \nabla_i\nabla_j u=& \nabla_ih_{jk}\nabla_k\Phi+h_{jk}\nabla_i\nabla_k\Phi\\
  =&\nabla_ih_{jk}\nabla_k\Phi+h_{jk}(\phi'g_{ik}-h_{ik}u)\\
  =&\nabla_kh_{ij}\nabla_k\Phi+\phi' h_{ij}-(h^2)_{ij}u,
\end{array}
\]
where in the last step we have used the fact that second fundamental form of hypersurfaces in space forms is Codazzi.
\end{proof}

\section{A flow of Mean curvature type}\label{sec flow}

Let $M(t)$ be a smooth family of closed hypersurfaces in a space form $N$. Let $X(\cdot,t)$ denote a point on $M(t)$. In general, we have the following evolution property.

\begin{lemm}\label{prop2}
Let $M(t)$ be a smooth family of closed hypersurfaces in $N^{n+1}(K)$ evolving along the flow
\[\partial_t X = f(X(\cdot,t))\nu,\] where $\nu$ is the unit outward normal vector field and $f$ is a function defined on $M(t)$. Then, we have the following evolution equations.

\begin{equation}
\begin{array}{rll}
\partial_t g_{ij} &=& 2fh_{ij}\\
\partial_t d\mu_g &=& fd\mu_g\\
\partial_t h_{ij} &=& -\nabla_i\nabla_j f + f (h^2)_{ij}-Kfg_{ij}\\
\partial_t H  &=& -\Delta f - f |h|^2-nKf,
\end{array}
\end{equation}
where $|h|^2$ is the norm square of second fundamental form and $d\mu_g$ is the volume element of the metric $g(t)$.
\end{lemm}
\begin{proof}
  Proof is standard, see for example, \cite{H}.
\end{proof}

Now, we consider the following mean curvature type of flow in $(N,ds^2)$ where $ds^2=d\rho^2+\phi^2(\rho)dz^2$,
\begin{equation}\label{uMCFh}
 \partial_tX =\di (n\phi'(\rho)-uH)\nu.
\end{equation}
Note that $\phi'(\rho)$ is $1$, $\cos(\rho)$, and $\cosh(\rho)$ in $\mathbb R^{n+1}$, $\mathbb S^{n+1}$, and $\mathbb H^{n+1}$ respectively.
\begin{prop}\label{mean curvature evolution}
  Let $M(t)$ be a smooth family of closed hypersurfaces in the space form $N^{n+1}(K)$ evolve along the flow (\ref{uMCFh}). Then the mean curvature evolves as the following

\begin{equation}
\begin{array}{rll}
\partial_t H &=& u\Delta H+H\nabla H\nabla \Phi(\rho)+2\nabla H\nabla u+ (H^2-n|h|^2)\phi'.
\end{array}
\end{equation}
\end{prop}

\begin{proof}
  From Lemma \ref{prop2}, we have

\begin{equation}
\begin{array}{rll}
\partial_t H &=& u\Delta H+2\nabla u\nabla H+H\Delta u-n\Delta \phi' -(n\phi'-Hu)(|h|^2+nK)\\
  &=&u\Delta H+2\nabla u\nabla H+H(\nabla H\nabla \Phi+H\phi'-|h|^2u)-n\Delta \phi' \\ &&-(n\phi'-Hu)(|h|^2+nK)\\
  &=&u\Delta H+2\nabla u\nabla H+H\nabla H\nabla \Phi+ (H^2-n|h|^2)\phi' \\
    &&-n(\Delta \phi' +K\Delta\Phi),
\end{array}
\end{equation}
where we have used (\ref{hessian of u}). It's easy to see that $\Delta \phi'=-K\Delta\Phi$ holds for all the three different $K$. This proves the assertion.
\end{proof}

The following estimate will be useful in later sections.

\begin{coro}
  Let $N^{n+1}(K)$ be hyperbolic space. Then along flow (\ref{uMCFh}), the mean curvature is uniformly bounded from above, i.e.,
  \begin{equation}
    H(\cdot, t)\le\max_{x\in M(0)} H(x,0).
    \label{H upper bound}
  \end{equation}
\end{coro}

\begin{proof}
  The proof is an immediate corollary of Proposition \ref{mean curvature evolution} by applying maximum principle to $H$. Notice that $\phi'=\cosh\rho>0$ in this case and $(H^2-n|h|^2)\le 0$ by Newton-McLaurin inequality.  
\end{proof}

\begin{rema}
  We need the uniform upper bound of mean curvature $H$ just in the hyperbolic case. 
\end{rema}

Applying the Minkowski identities, we have the following monotonicity formulas,

\begin{prop}
  \label{prop mono}
  Along the flow (\ref{uMCFh}), the volume $V(t)$ of the enclosed domain by the closed hypersurface is a constant and surface area evolves as follows,
  \begin{equation}
    \begin{array}[]{rll}
      \di \frac{d}{dt}A(t)      =& -\di\frac{1}{n-1}\Dint_{M(t)}\sum_{i< j}(\kappa_i(x)-\kappa_j(x))^2ud\mu_g,\\
    \end{array}
    \label{area evolution}
  \end{equation}
  where $\kappa(x)$ are the principal curvatures of the hypersurface at point $x\in M(t)$.
\end{prop}
\begin{proof}
  In general, along flow equation $\P_t X= f\nu$, the volume evolution is $V'(t)= \Dint_{M(t)}fd\mu_g$ and surface area satisfies $A'(t)=\Dint_{M(t)}fd\mu_g$. Thus, by Minkowski identity (\ref{Minkowski}),
  \[
  V'(t) = \Dint_M (n\phi'-Hu)d\mu =0,
  \]
  By the general Minkowski identity (\ref{general Minkowski}), we have
  \[
  \begin{array}[]{rll}
  A'(t) =&\Dint_{M}(n\phi'-Hu)Hd\mu_g\\
  =&\Dint_{M}(n\phi'H-\frac{2n}{n-1}\sigma_2(\kappa)u)d\mu+\Dint_M(\frac{2n}{n-1}\sigma_2-H^2)ud\mu\\
  =&-\Dint_M(H^2-\frac{2n}{n-1}\sigma_2)ud\mu\\
  =&-\Dint_M(\sum_i\kappa_i^2-\frac{2}{n-1}\sigma_2)ud\mu\\
  =&-\di\frac{1}{n-1}\Dint_M\sum_{i<j}(\kappa_i-\kappa_j)^2ud\mu.
  \end{array}
  \]
\end{proof}

Next we consider the special case of a bounded convex domain in $\mathbb R^{n+1}$. We first show that the weak convexity of the initial hypersurface is preserved and improved along the flow (\ref{umcf}).

\begin{lemm}
  \label{convexity}
  Let $M_0$ be the smooth boundary of a bounded convex domain in $\mathbb R^{n+1}$ which may have flat side. If $M_t$ is a solution of flow (\ref{umcf}) with initial hypersurface to be $M_0$, then $M_t$ is strictly convex for any $t>0$ whenever the flow exists.
\end{lemm}
\begin{proof}
  The proof follows immediately from Theorem 1.4 of Bian-Guan \cite{BG}.
\end{proof}

\begin{prop}
  \label{prop mono quermass}
 Let $M_0$ be the smooth boundary of a bounded convex domain in $\mathbb R^{n+1}$. If $M_t$ is a solution of flow (\ref{umcf}) with initial hypersurface to be $M_0$, then for any $k=0,1,\cdots, n-2$,

 \begin{equation}
    \begin{array}[]{rll}
      \di \frac{d}{dt}\Dint_{M}\sigma_k(\kappa)d\mu      \le 0 ,\\
    \end{array}
    \label{}
  \end{equation}
  where $\kappa(x)$ are the principal curvatures of the hypersurface at point $x\in M(t)$.
\end{prop}

\begin{proof}
  Recall for a general evolution equation,
  \[
  \P_t X = F\nu,
  \]
  with speed function $F(X, t)$, we have the evolution equation of the qumermassintegrals,  c.f. Proposition 4 of \cite{GL},
  \begin{equation}
  \di \frac{d}{dt}\Dint_{M}\sigma_k(\kappa)d\mu      = (k+1)\Dint_M F\sigma_{k+1}(\kappa)d\mu .\\
  \label{mono quermass equ1}
\end{equation}

  For any $t>0$, $\sigma_l>0$ for all $l=0,1,\cdots,n$ by Lemma \ref{convexity}. By Newton-McLaurin inequality, we have the following,
  \begin{equation}
    \sigma_1\sigma_{k+1}\ge c_1\sigma_{k+2}.
  \label{mono quermass equ2}
  \end{equation}
  where $c_1=\frac{n(k+2)}{n-k-1}$.
  Let $F=n-Hu$ in (\ref{mono quermass equ1}) and apply inequality (\ref{mono quermass equ2}), we have
\[
\begin{array}[]{rll}
  \di \frac{d}{dt}\Dint_{M}\sigma_k(\kappa)d\mu      =& (k+1)\Dint_M (n\sigma_{k+1}-u\sigma_1\sigma_{k+1})d\mu \\
  \le &(k+1) \Dint_M (n\sigma_{k+1}-c_1u\sigma_{k+2})d\mu\\
  =&0,
\end{array}
\]
where in the last step we used Minkowski identity (\ref{general Minkowski}). This completes the assertion.
\end{proof}

\section{Radial graph and $C^0$ estimate}\label{sec C0}

Let $(M,g)$ be a hypersurface in $N^{n+1}(K)$ with induced metric $g$. We now give the local expressions of the induced metric, second fundamental form, Weingarten curvatures etc when $M$ is a graph of a smooth and positive function $\rho(z)$ on $\mathbb S^n$. Let $\P_1, \cdots, \P_n$ be a local frame along $M$ and $\P_\rho$ be the vector field along radial direction. Then the support function, induced metric, inverse metric matrix, second fundamental form can be expressed as follows (\cite{O}). For simplicity, all the covariant derivatives with respect to the standard spherical metric $e_{ij}$ will also be denoted as $\nabla$ when there is no confusion in the context.
\begin{equation}\label{tensors:rho}
    \begin{array}{rll}
    u=&\frac{\phi^2}{\sqrt{\phi^2+|\nabla \rho|^2}}\\
    g_{ij}=& \phi^2\delta_{ij}+\rho_i\rho_j,\quad g^{ij}=\frac{1}{\phi^2}(e^{ij}-\frac{\rho_i\rho_j}{\phi^2+|\nabla \rho|^2})\\
    h_{ij}=&(\sqrt{\phi^2+|\nabla \rho|^2})^{-1}(-\phi\nabla_i\nabla_j\rho+2\phi'\rho_i\rho_j+\phi^2\phi'e_{ij})\\
    h^i_j=& \frac{1}{\phi^2\sqrt{\rho^2+|\nabla \rho|^2}}(e^{ik}-\frac{\rho_i\rho_k}{\phi^2+|\nabla \rho|^2})(-\phi\nabla_k\nabla_j\rho+2\phi'\rho_k\rho_j+\phi^2\phi'e_{kj})
    \end{array}
\end{equation}
where all the covariant derivatives $\nabla$ and $\rho_i$ are w.r.t. the spherical metric $e_{ij}$. \\

It will be convenient if we introduce a new variable $\gamma$ satisfying
\[
\frac{d\gamma}{d\rho}=\frac{1}{\phi(\rho)}.
\]

Let $\omega:=\sqrt{1+|\nabla\gamma|^2}$, one can compute the unit outward normal $\nu= \frac{1}{\omega}(1,-\frac{\rho_1}{\phi^2},\cdots, -\frac{\rho_n}{\phi^2})$ and the general support function $u=\<D\phi, \nu\>= \frac{\phi}{\omega}$. Moreover,
\begin{equation}\label{tensors:gamma}
    \begin{array}{rll}
g_{ij}=& \phi^2\delta_{ij}+\rho_i\rho_j,\quad g^{ij}=\frac{1}{\phi^2}(e^{ij}-\frac{\gamma_i\gamma_j}{\omega^2})\\
    h_{ij}=&\frac{\phi}{\omega}(-\gamma_{ij}+\phi'\gamma_i\gamma_j+\phi'e_{ij})\\
    h^i_j=& \frac{1}{\phi\omega}(e^{ik}-\frac{\gamma_i\gamma_k}{\omega^2})(-\gamma_{kj}+\phi'\gamma_k\gamma_j+\phi'e_{kj})
    \end{array}
\end{equation}
where all the covariant derivatives are w.r.t. the spherical metric $e_{ij}$. \\

We also have for the mean curvature,
\begin{equation}
    H= \frac{1}{\phi\omega}(e^{ij}-\frac{\gamma_i\gamma_j}{\omega^2})(-\gamma_{ij}+\phi'\gamma_i\gamma_j+\phi'e_{ij}).
  \label{mean curvature}
\end{equation}

We now consider the flow equation (\ref{uMCFh}) of radial graphs over $\mathbb S^n$ in $N^{n+1}(K)$. It is known (\cite{G1}) if a closed hypersurface which is a radial graph and satisfies
\[
\P_t X= f\nu,
\]
then the evolution of the scalar function $\rho=\rho(X(z,t),t)$ satisfies\[
\P_t\rho = f\omega.
\]

Thus we only need to consider the following parabolic initial value problem on $\mathbb S^n$,
\begin{equation}
  \left\{
  \begin{array}[]{rll}
    \P_t\rho =& (n\phi'-Hu)\omega, {\rm for\quad}(z,t)\in \mathbb S^n\times [0,\infty)\\
    \rho(\cdot, 0)=&\rho_0,
  \end{array}
  \right.
  \label{ivp}
\end{equation}
where $\rho_0$ is the radial function of the initial star-shaped hypersurface.

Equivalently, the equation for $\gamma$ satisfies

\begin{equation}
  \begin{array}[]{rll}
    \P_t\gamma =& (n\phi'-Hu)\frac{\omega}{\phi}  \end{array}
  \label{ivp gamma}
\end{equation}

We next show that the radial function $\rho$ is uniformly bounded from above and below.
\begin{prop}
  Let $M_0$ be a radial graph of function $\rho_0$ embedded in space form $N^{n+1}(K)$. If $\rho(z,t)$ solves the initial value problem (\ref{ivp}), then for any $(z,t)\in \mathbb S^n\times [0, T]$,
\[\min_{x\in M} \rho(x,0)\le \rho(x, t)\le \max_{x\in M} \rho(x,0).\]
\end{prop}
\begin{proof}
  Note $\nabla\gamma = \frac{1}{\phi}\nabla \rho$, thus $\gamma$ attains maximum or minimum at the same point as $\rho$. At critical points of $\gamma$ and $\rho$, we have the following critical point conditions,
 \[\nabla \gamma=0, \nabla \rho=0, \omega =1.\]

  Then it follows from (\ref{tensors:gamma}) that, at critical points of $\gamma$ and $\rho$, $H=\frac{1}{\phi}(-\Delta \gamma +n\phi')$. Together with (\ref{ivp gamma}), at critical points,
  \[ \gamma_t=\frac{1}{\phi}\Delta \gamma, \quad{\rm and}\quad \P_t\rho=\frac{1}{\phi}\Delta\rho.\]
By standard maximum principle, this proves the uniform upper and lower bounds for $\gamma$ and $\rho$.
\end{proof}

\section{Gradient estimate and convergence }\label{sec estimates}
We need the following technical lemma.

\begin{lemm}\label{tech lemm}
  Let $\gamma$ be a solution of (\ref{ivp gamma}) and assume $|\nabla\gamma|^2$ attains maximum value at point $P$, then at $P$, the following holds,
  \begin{equation}
    \mathcal L(\frac{|\nabla\gamma|^2}{2})=n\frac{\phi''\phi-(\phi')^2}{\phi\omega}|\nabla\gamma|^4-\frac{n-1}{\phi\omega}|\nabla\gamma|^2-\frac{1}{\phi\omega}|\nabla\nabla\gamma|^2-\frac{\phi'}{\phi\omega}\Delta \gamma |\nabla\gamma|^2,
    \label{tech lemm equ}
  \end{equation}
  where $\mathcal L $ is a strictly parabolic operator defined as follows: for any function $f(z,t)$,
  \[
  \mathcal L(f):=f_t-\frac{1}{\phi\omega}(e^{ij}-\frac{\gamma_i\gamma_j}{\omega^2})f_{ij}
  \]
\end{lemm}

\begin{proof}
By critical point condition, at $P$ we have
 \[
 \nabla |\nabla\gamma|^2=0, \nabla\omega =0.
 \]
 From (\ref{tensors:gamma}) and (\ref{ivp gamma}), using critical point condition, we have
\[
\P_t\gamma = n\omega \frac{\phi'}{\phi}-\frac{1}{\phi\omega}(-\Delta\gamma +\frac{\gamma_i\omega_i}{\omega}+n\phi').
\]

Thus, at $P$
\begin{equation}
\begin{array}[]{rll}
  \P_t\frac{|\nabla\gamma|^2}{2} = & \gamma_k\nabla_k\gamma_t\\
  =&\gamma_k\nabla_k\big(n\omega \frac{\phi'}{\phi}-\frac{1}{\phi\omega}(-\Delta\gamma +\frac{\gamma_i\omega_i}{\omega}+n\phi')\big)\\
  =&\gamma_k\nabla_k\big(\frac{n\phi'}{\phi\omega}|\nabla\gamma|^2+\frac{\Delta\gamma}{\phi\omega}-\frac{\gamma_i\omega_i}{\omega^2}\big)\\
  =&n\frac{\phi''\phi-(\phi')^2}{\phi\omega}|\nabla\gamma|^4-\frac{\phi'}{\phi\omega}\Delta \gamma |\nabla\gamma|^2+\frac{1}{\phi\omega}\nabla_k\Delta \gamma\nabla_k\gamma-\frac{\omega_{ki}\gamma_i\gamma_k}{\phi\omega^2}\\
\end{array}
\label{gradient equ1}
\end{equation}
Recall on sphere, tensor $A_{ij}:=f_{ij}+fe_{ij}$ is Codazzi for any function $f(z)$, where $e_{ij}$ is the standard spherical metric and the covariant derivatives are w.r.t. $e_{ij}$. Thus
\[
  \nabla_i(\gamma_{jk}+\gamma e_{jk}) = \nabla_k(\gamma_{ij}+\gamma e_{ij}).
\]
Applying this property, we have
\begin{equation}
  \begin{array}[]{rll}
  \nabla_k\Delta\gamma\nabla_k\gamma=& \nabla_i\gamma_{ki}\nabla_k\gamma-(n-1)|\nabla\gamma|^2 \\
  =&\Delta\frac{|\nabla\gamma|^2}{2}-|\nabla\nabla\gamma|^2-(n-1)|\nabla\gamma|^2.
  \end{array}
  \label{gradient equ2}
\end{equation}

Plug (\ref{gradient equ2}) into (\ref{gradient equ1}), we have
\begin{equation}
\begin{array}[]{rll}
  \P_t\frac{|\nabla\gamma|^2}{2} = & \frac{1}{\phi\omega}\Delta \frac{|\nabla\gamma|^2}{2}-\frac{\omega_{ik}\gamma_i\gamma_k}{\phi\omega^2}\\
  &+ n\frac{\phi''\phi-(\phi')^2}{\phi\omega}|\nabla\gamma|^4-\frac{n-1}{\phi\omega}|\nabla\gamma|^2-\frac{1}{\phi\omega}|\nabla\nabla\gamma|^2-\frac{\phi'}{\phi\omega}\Delta \gamma |\nabla\gamma|^2\\
\end{array}
\label{gradient equ3}
\end{equation}

Note that
\[
\frac{1}{\phi\omega}(e^{ij}-\frac{\gamma_i\gamma_j}{\omega^2})(\frac{|\nabla\gamma|^2}{2})_{ij}=\frac{1}{\phi\omega}\Delta \frac{|\nabla\gamma|^2}{2}-\frac{\omega_{ik}\gamma_i\gamma_k}{\phi\omega^2}.
\]
This completes the proof.
\end{proof}

\subsection{ $\mathbb R^{n+1}$ and $\mathbb S^{n+1}$ cases}
In this subsection, we give a priori estimates for $\mathbb R^{n+1}$ and $\mathbb S^{n+1}$ cases which are relatively easier.

\begin{prop} \label{prop R and S}
   Let $M_0$ be a radial graph of function $\rho_0$ over $\mathbb S^n$ in space form $N^{n+1}(K)$ for $K=0$ or $+1$. If $\rho(z,t)$ solves the initial value problem (\ref{ivp}) on interval $[0,T]$, then for any $(z,t)\in \mathbb S^n\times [0, T]$,
\[|\nabla \rho|(z,t)\le C\]
where $C$ is a uniform constant only depends on the initial values and the covariant derivatives are w.r.t. the spherical metric on $\mathbb S^n$.
Moreover, the following exponential convergence holds, for any $(z,t)\in \mathbb S^n\times [0, T]$,
\begin{equation}
e^{\alpha t}|\nabla\gamma|^2(z,t)\le \di\max_{z\in \mathbb S^n}|\nabla\gamma|^2(z,0),
\label{prop exp}
\end{equation}
where $\alpha>0$ is a uniform constant depending only on the initial graph.
\end{prop}

\begin{proof}
  By Lemma \ref{tech lemm}, we have at the maximal point $P$,
  \begin{equation}
   \begin{array}[]{rll}
    \mathcal L(\frac{|\nabla\gamma|^2}{2})=&n\frac{\phi''\phi-(\phi')^2}{\phi\omega}|\nabla\gamma|^4-\frac{n-1}{\phi\omega}|\nabla\gamma|^2-\frac{1}{\phi\omega}|\nabla\nabla\gamma|^2-\frac{\phi'}{\phi\omega}\Delta \gamma |\nabla\gamma|^2\\
    \le&n\frac{\phi''\phi-(\phi')^2}{\phi\omega}|\nabla\gamma|^4-\frac{n-1}{\phi\omega}|\nabla\gamma|^2-\frac{1}{n\phi\omega}(\Delta\gamma)^2-\frac{\phi'}{\phi\omega}\Delta \gamma |\nabla\gamma|^2,
  \end{array}
    \label{gradient equ0.0}
  \end{equation}
  where we have used Cauchy-Schwartz inequality in the last step. By completing the square, we obtain
  \begin{equation}
     \begin{array}[]{rll}
    \mathcal L(\frac{|\nabla\gamma|^2}{2})\le&n\frac{\phi''\phi-(\phi')^2}{\phi\omega}|\nabla\gamma|^4-\frac{n-1}{\phi\omega}|\nabla\gamma|^2\\
    &-\frac{1}{n\phi\omega}\big(\Delta\gamma+\frac{n}{2}\phi' |\nabla\gamma|^2\big)^2 +\frac{n}{4\phi\omega}(\phi')^2|\nabla\gamma|^4\\
   = &-\frac{n}{\phi\omega}(\frac{3}{4}(\phi')^2-\phi''\phi)|\nabla\gamma|^4-\frac{n-1}{\phi\omega}|\nabla\gamma|^2 -\frac{1}{n\phi\omega}\big(\Delta\gamma+\frac{n}{2}\phi' |\nabla\gamma|^2\big)^2
  \end{array}
    \label{gradient equ0}
  \end{equation}

  In $\mathbb R^{n+1}$, $\frac{3}{4}(\phi')^2-\phi''\phi=\frac{3}{4}$. In $\mathbb S^{n+1}$, $\frac{3}{4}(\phi')^2-\phi''\phi=\frac{3}{4}+\frac{1}{4}\sin^2(\rho)$. Thus in both these two cases, we have
  \begin{equation}
     \begin{array}[]{rll}
    \mathcal L(\frac{|\nabla\gamma|^2}{2})\le-\frac{n-1}{\phi\omega}|\nabla\gamma|^2\le 0.   \end{array}
    \label{gradient equ4}
  \end{equation}

  Standard parabolic maximal principle yields that $|\nabla\gamma|^2(z,t)\le C= \max |\nabla \gamma|^2(z,0)$. Since $|\nabla\gamma|=\frac{|\nabla\rho|}{\phi} $, and $\phi$ is uniformly bounded from above and below by positive constants, we have the uniform bound for $|\nabla\rho|$ as well.

  Next, we prove the exponential estimate. Since we have shown that $|\nabla\gamma|^2$ is uniformly bounded from above, so is $\omega$. Thus there exists a uniform constant $\alpha \ge \frac{2(n-1)}{\phi\omega}>0$ which depends only on the upper bound of $\omega$ and $\phi$. By (\ref{gradient equ4}),

  \[
     \begin{array}[]{rll}
       \mathcal L(\frac{|\nabla\gamma|^2}{2})\le-\alpha\frac{|\nabla\gamma|^2}{2}.
     \end{array}
  \]
  Standard maximal principle then proves (\ref{prop exp}).

\end{proof}

\subsection{a Priori estimates : $\mathbb H^{n+1}$ case}
In this section, we derive gradient estimate and exponential convergence of the parabolic initial value problem (\ref{ivp}) in $\mathbb H^{n+1}$. The difficulty in $\mathbb H^{n+1}$ is because in (\ref{gradient equ0}), $\frac{3}{4}(\phi')^2-\phi''\phi=\frac{3}{4}-\frac{1}{4}\sinh^2(\rho)$ and a priorily, it may not be positive for all initial graph. Since this term is the leading term, we have to find other means to get around.

To obtain a uniform gradient bound, we have to use the fact that the mean curvature is uniformly bounded from above.

\begin{prop}
  \label{gradient}
   Let $M_0$ be a radial graph of function $\rho_0$ over $\mathbb S^n$ in space form $\mathbb H^{n+1}$. If $\rho(z,t)$ solves the initial value problem (\ref{ivp}) on interval $[0,T]$, then for any $(z,t)\in \mathbb S^n\times [0, T]$,
\[|\nabla \rho|(z,t)\le C\]
where $C$ is a uniform constant only depends on the initial values and the covariant derivatives are w.r.t. the spherical metric on $\mathbb S^n$.
\end{prop}

\begin{proof}
  As in the proof of Proposition \ref{prop R and S}, we carry out computations at the maximum point $P$ of the test function $|\nabla\gamma|^2$. Recall in (\ref{gradient equ0.0}), we have obtained
  \begin{equation}
    \mathcal L(\frac{|\nabla\gamma|^2}{2})=n\frac{\phi''\phi-(\phi')^2}{\phi\omega}|\nabla\gamma|^4-\frac{n-1}{\phi\omega}|\nabla\gamma|^2-\frac{1}{\phi\omega}|\nabla\nabla\gamma|^2-\frac{\phi'}{\phi\omega}\Delta \gamma |\nabla\gamma|^2\\
    \label{gradient H equ1}
  \end{equation}

    On the other hand, from (\ref{mean curvature}), at the critical point $P$, we have
    \[
      H=\frac{1}{\phi\omega}(-\Delta \gamma + n\phi').
      \]
Or equivalently,
\begin{equation}
      - \frac{1}{\phi\omega}\Delta \gamma=H -\frac{n\phi'}{\phi\omega}.
  \label{gradient H equ2}
\end{equation}

      Plug (\ref{gradient H equ2}) into (\ref{gradient H equ1}), after simplifications, we have

      \[
        \begin{array}[]{rll}
    \mathcal L(\frac{|\nabla\gamma|^2}{2})=&n\frac{\phi''\phi-(\phi')^2}{\phi\omega}|\nabla\gamma|^4-\frac{1}{\phi\omega}\big(n-1+n(\phi')^2-\phi\phi'\omega H\big)|\nabla\gamma|^2-\frac{1}{\phi\omega}|\nabla\nabla\gamma|^2\\
    =& -\frac{1}{\phi\omega}\big[\omega(n\omega-H\cosh\rho\sinh\rho)+n\cosh^2\rho-1\big]|\nabla\gamma|^2-\frac{1}{\phi\omega}|\nabla\nabla\gamma|^2,
        \end{array}
        \]
        where we have used $\omega^2={1+|\nabla\gamma|^2}$, $\phi=\sinh\rho$, and $\phi''\phi-(\phi')^2=-1$.
        At $P$, if $n\omega\le H\cosh\rho\sinh\rho$, by the uniform upper bound of $H$, (c.f. Corollary \ref{H upper bound}), and the uniform upper bound of $\rho$, we have $\omega\le C$ and the proof is done. Now we can assume that $n\omega-H\cosh\rho\sinh\rho>0$. Note $\cosh\rho\ge1$, thus

        \begin{equation}
        \begin{array}[]{rll}
    \mathcal L(\frac{|\nabla\gamma|^2}{2})  =& -\frac{1}{\phi\omega}\big[\omega(n\omega-H\cosh\rho\sinh\rho)+n\cosh^2\rho-1\big]|\nabla\gamma|^2-\frac{1}{\phi\omega}|\nabla\nabla\gamma|^2\\
    \le& 0.
        \end{array}
          \label{gradient H equ3}
        \end{equation}
By the standard maximal principle, we prove the assertion.
\end{proof}

\subsection{Long-time existence and exponential convergence}

\bigskip
By direct simplification, we obtain the following lemma by (\ref{tensors:gamma}), (\ref{mean curvature}), and (\ref{ivp gamma}).
\begin{lemm}
  \label{divergent}
Recall $\rho(z)$, $z\in\mathbb S^n$, is a positive function on unit sphere and $\phi(\rho)$ is $\sin(\rho)$, $\rho$, or $\sinh(\rho)$ when $K=+1,0, -1$ respectively. Equation (\ref{ivp gamma}) is then a parabolic equation of divergent form, i.e.,
  \[\gamma_t=\mathrm{div} \big(\frac{1}{\phi \omega}\nabla \gamma\big)+\frac{2\phi' |\nabla\gamma|^2}{\phi \omega},\]
  where $\omega=\sqrt{1+|\nabla\gamma|^2}$, and $\gamma(z):=\gamma(\rho(z))$ satisfies $\frac{\P\gamma}{\P \rho}=\frac{1}{\phi(\rho)}$.  \\

Equivalently, $\rho(z,t)$ solves the following divergent parabolic equation,
\[\rho_t=\mathrm{div} \big(\frac{1}{\phi \tilde \omega}\nabla \rho\big)+\frac{2\phi' |\nabla\rho|^2}{\phi^2 \tilde\omega},\]
where $\tilde \omega = \sqrt{\phi^2+|\nabla\rho|^2}$ and all the covariant derivatives are w.r.t. the standard metric on unit sphere.
\end{lemm}

\bigskip

\noindent
\begin{proof} {\bf( Proof of Theorem \ref{main thm} for cases : $\mathbb R^{n+1}$ and $\mathbb S^{n+1}$.)}
Since (\ref{ivp}) or (\ref{ivp gamma}) is a divergent quasi-linear parabolic equation, by the classical theory of parabolic equation in divergent form (e.g., \cite{Lady}), the higher regularity a priori estimates of the solution follow from the uniform gradient estimates in Proposition \ref{prop R and S} and Proposition \ref{gradient}. Moreover, the solution of $\rho(\cdot, t)$ exists for all $t\in[0,\infty)$. Now we are ready to prove our main theorem for the Euclidean and elliptic cases, that is the exponential convergence. This follows from the exponential estimate (\ref{prop exp}) in Proposition \ref{prop R and S}. That is,  $e^{\beta t}|\nabla\rho(t)|\rightarrow 0$ as $t\rightarrow \infty$ for each $\beta<\alpha$, and the graph converges to a sphere exponentially fast. This proves our main theorem for the Euclidean and elliptic spaces. \end{proof}

\medskip

What's left is the exponential convergence in hyperbolic space $\mathbb H^{n+1}$. Since the gradient estimate implies uniform bounds for all the a priori estimates because the equation is quasilinear divergent parabolic, we now know that the principal curvatures, surface area element, and other geometric quantities are all uniformly bounded from above.

Next, we will show that when $t$ is large enough, different principal curvatures of the radial graph at the same point are comparable uniformly for arbitrary small $\epsilon$.

\begin{prop}
\label{eigenvalues}
  Let $\rho(z,t)$ solves the initial value problem (\ref{ivp}). Then for any $\epsilon>0$, there exists a $T_0>0$, such that for any $t>T_0$,
  \begin{equation}
    \di\max_{i< j, z\in \mathbb S^n}|\kappa_i(z,t)-\kappa_j(z,t)|<\epsilon,
    \label{}
  \end{equation}
  where $\kappa(z,t)$ are the principal curvatures of the radial graph $\rho(z,t)$ at $(z,t)\in \mathbb S^n\times [0,\infty)$.
\end{prop}
\begin{proof}

 Let $A(t)$ be the surface area of the radial graph at time $t>0$, then
  \begin{equation}
    A(t)-A(0)=\Dint^t_0 A'(s)ds
    \label{integral}
  \end{equation}

  Since $A'(s)\le 0$ from Proposition \ref{prop mono} and also $A(t)$ is uniformly bounded from above and below, by our regularity estimates, we have $\Dint^\infty_0 A'(t)dt<\infty$. By Proposition \ref{prop mono},
  \[
    |A'(t)|=-\frac{1}{n-1}\Dint_{M(t)}\sum_{i< j}(\kappa_i(x)-\kappa_j(x))^2ud\mu_g.
    \]
Since we have established the uniform a priori estimates for all the derivatives of $\rho$ both with respect to $z$ and $t$ for any order and also $0<C_1\le u\le C_2$ uniformly, $A'(t)$ is uniformly continuous with respect to $t$ and $d\mu_g$ is also uniformly bounded from both up and below. This proves the assertion when $T$ is big enough.
\end{proof}

Applying Proposition \ref{eigenvalues}, we will prove the following exponential convergence theorem which proves that the solution of the initial value problem (\ref{ivp}) for the hyperbolic case also converges to a standard sphere exponentially.

\begin{prop}
Let $M_0$ be a radial graph of function $\rho_0$ over $\mathbb S^n$ in space form $\mathbb H^{n+1}$. If $\gamma(z,t)$ solves (\ref{ivp gamma}) on interval $[0,T]$, then there exist a uniform constant $\alpha>0$ which depends only on the initial graph, such that for any $(z,t)\in \mathbb S^n\times [0, T]$,
\[e^{\alpha t}|\nabla \gamma|^2(z,t)\le \di\max_{z\in \mathbb S^n}|\nabla \gamma|^2(z,0)\]
where the covariant derivatives are w.r.t. the spherical metric on $\mathbb S^n$.
\end{prop}

\begin{proof}
  Let $\epsilon>0$ be a small enough constant to be determined and $T_0>0$ be defined as in Proposition \ref{eigenvalues}. For any $T>T_0$, assume $\frac{|\nabla\gamma|^2}{2}$ achieves maximum on the interval $[T_0,T]$ at point $P$. Recall in the proof of Proposition \ref{gradient}, we have proved (\ref{gradient H equ3}) at $P$,
  \begin{equation}
\begin{array}{rll}
 \mathcal L (\frac{|\nabla\gamma|^2}{2})
&\le&-\frac{1}{\omega\sinh\rho}(n\omega^2-\omega\cosh\rho \sinh\rho H+n\cosh^2\rho-1)|\nabla \gamma|^2.\\
\end{array}
\label{exp equ1}
\end{equation}

At the critical point $P$, we can simplify the Weingarten tensor $h^i_j=\frac{1}{\omega\sinh\rho}(-\gamma_{ij}+\cosh\rho\delta_{ij})$, where $\delta_{ij}$ is the standard spherical metric if we choose orthonormal coordinates at this point. By rotating the coordinates, we can assume $\gamma_1=|\nabla \gamma|$ and $\gamma_j=0$ for $2\le j\le n$. Then by critical point condition, $\gamma_{1j}=0$ for all $j=1,\cdots,n$. Now fixing the $i=1$ direction, we can diagonalize the rest $(n-1)\times (n-1)$ matrix. Finally, $\gamma_{ij}$ becomes a diagonal matrix with $\gamma_{11}=0$. Since the Weingarten tensor $h^i_j=\frac{1}{\omega\sinh\rho}(-\gamma_{ij}+\cosh\rho\delta_{ij})$ at $P$, it is also diagonalized so the eigenvalues are $h^i_i$ which are also the principal curvatures. Thus at $P$, the principal curvatures
\begin{equation}
 \kappa_i=h^i_i= \frac{1}{\omega\sinh\rho}(-\gamma_{ii}+\cosh\rho).
 \label{kappa}
  \end{equation}
  By Proposition \ref{eigenvalues} and (\ref{kappa}), we have
  \begin{equation}
  |\gamma_{ii}-\gamma_{jj}|\le \epsilon C_1
  \label{gamma}
\end{equation}
where $C_1$ is a uniform constant and we have used the uniform boundedness of $\rho$ and $\omega$. Since $\gamma_{11}=0$, using (\ref{gamma}), we have $|\gamma_{jj}|=|\gamma_{11}-\gamma_{jj}|\le \epsilon C_1$ for $2\le j\le n$. Recall at $P$,
\[
 - \omega \sinh\rho H=\sum \gamma_{ii}-n\cosh\rho\ge -(n-1)\epsilon C_1-n\cosh\rho.
  \]
For small enough $\epsilon$ which depends on the uniform bounds of $\rho$, we have
 \begin{equation}
   \begin{array}[]{rll}
   &n\omega^2-\omega\cosh\rho \sinh\rho H+n\cosh^2\rho-1\\
   \ge &(n-1)+(n-(n-1)\epsilon C_1\cosh\rho)|\nabla\gamma|^2\\
   \ge& n-1.
   \end{array}
   \label{exp equ2}
 \end{equation}

 Now if we plug (\ref{exp equ2}) into (\ref{exp equ1}), we have

 \begin{equation}
   \mathcal L (\frac{|\nabla\gamma|^2}{2})
\le -\frac{n-1}{\omega\sinh\rho}|\nabla\gamma|^2\le -\alpha \frac{|\nabla\gamma|^2}{2},
   \label{}
 \end{equation}
 where $\alpha>0$ is a constant independent of time $t$. Standard maximal principle proves that
 \[
    \di\max_{(z,t)\in \mathbb S^n\times [T_0,T]} e^{\alpha t}|\nabla \gamma|^2\le C_1.
    \label{}
   \]
   Since $ \di\max_{(z,t)\in \mathbb S^n\times [0,T_0]} e^{\alpha t}|\nabla \gamma|^2\le e^{\alpha T_0}C_2$, where $C_2$ is the uniform upper bound of $|\nabla\gamma|^2$, we have

 \[
   \di\max_{(z,t)\in \mathbb S^n\times [0,T]} e^{\alpha t}|\nabla \gamma|^2\le C,
    \label{}
   \]
where $C$ is a uniform constant depending on $C_1$, $C_2$, and $T_0$. Since all the constants including $T_0$ are uniform, we finish the proof.

\end{proof}

\begin{proof} {\bf( Proof of Theorem \ref{main thm} : $\mathbb H^{n+1}$.)}
With all the a priori estimates and exponential convergence established, the proof is now standard.
\end{proof}

\section{Preserving convexity}

In this section, we prove flow (\ref{umcf}) preserves convexity in $\mathbb R^{n+1}$. The proof follows arguments in the proof of Theorem 1.4 in Bian-Guan \cite{BG}. Since the statement in Theorem 1.4 in \cite{BG} does not imply directly the convexity of flow (\ref{umcf}) (even though its proof does), we state the following general result. We denote $\mathcal{S^{+}}_{n}$ to be the set of all positive definite $n\times n$ matrices, $u=X\cdot \nu$, $g$ and $h$ the first and the second fundamental forms of surface $M$ respectively. The proof in \cite{BG} can be adapted by considering $u$ as an independent variable and making use of Lemma \ref{hessian of u} in the case of $N=\mathbb R^{n+1}$.

\begin{theo}\label{thmw2-flow-1}
Suppose $F(A,X, \nu, u, t)$ is monotone (i.e., $\frac{\partial F}{\partial A_{ij}}>0$) in $A$ and $F(A^{-1},X,\nu, u, t)$ is locally convex in $(A, X)\in
 \mathcal{S^{+}}_{n}\times \mathbb R^{n+1}$ for each fixed $\nu \in \mathbb S^n$, $u\in \mathbb R^{+}$,
$t\in [0,T]$ for some $T>0$. Let $M(t)$ be an oriented immersed
connected hypersurface in $ \mathbb  R^{n+1}$ with a nonnegative
definite second fundamental form $h(t)$ satisfying equation
\begin{equation}\label{flow1} X_t=-F(g^{-1}h, X, \nu, u, t)\nu.\end{equation}
Then $h(t)$ is of constant rank $l(t)$ for each
$t\in(0,T]$ and $l(s)\le l(t)$ for all $0<s\le t\le T$. Moreover
 the null space of $h$ is parallel for
each $t$. In particular, if $M_0$ is compact and
convex, then $M(t)$ is strictly convex for all $t\in (0,T)$.
\end{theo}

\noindent{\bf Proof.} The proof follows the similar lines of arguments in the proof of
Proposition 5.1 in \cite{BG}, here we will just indicate some modifications needed when $u$ is considered as an independent variable. We use the same notations as in \cite{BG}.  For
$\epsilon>0$, let $W=(g^{im}h_{mj}+\epsilon \delta_{ij}) $, let $l(t)$ be the minimal rank of $h(t)$. For a fixed
$t_0\in (0,T)$, let $x_0\in M$ such that $h(t_0)$ attains minimal
rank at $x_0$. Set
$\phi(x,t)=\sigma_{l+1}(W(x,t))+\frac{\sigma_{l+2}}{\sigma_{l+1}}(W(x,t))$.
It is proved in section 2 in \cite{BG} that $\phi$ is in $C^{1,1}$. The first part of Theorem will follow if we can establish that there are constants $C_1, C_2$ independent of $\epsilon$ such that
\begin{eqnarray}\label{flow3}
F^{ij}\phi_{ij}-\phi_t\le C_1\phi+C_2|\nabla \phi|, \quad \mbox{near $(x_0,t_0)$.}\end{eqnarray}

Denote $h^2=(h^i_lh^l_j)$.
Note that under (\ref{flow1}), the Weingarten
form $h^i_j=g^{im}h_{mj}$ satisfies the equation
\begin{equation}\label{flow2} \partial_t h^i_j  =
\nabla^i\nabla_j F + F (h^2)^i_j.\end{equation}

As in \cite{BG}, near $(t_0, x_0)$, the index set $\{1,\cdots, n\}$ can be divided in to two subset $B, G$,
where for $i\in B$, the eigenvalues of $(W_{ij})$, $\lambda_{i}$ is small and for $j\in G$, $\lambda_{j}$ is strictly positive away from $0$.
As in \cite{BG}, we may assume at each point of computation, $(W_{ij}$ is diagonal. Notice that $W_{ii}\le \phi$ for all $i\in B$, so we have the following inequality as corresponding to inequality (5.3) in \cite{BG}:
\begin{eqnarray*}
&& \sum F^{\alpha\beta}\phi_{\alpha\beta}-\phi_t=O(\phi +\sum_{i,j\in
B}|\nabla W_{ij}|)-\frac{1}{\sigma_{1}(B)} \sum_{\alpha,\beta}
\sum_{i,j\in
B,i\neq j}F^{\alpha\beta}W_{ij\alpha}W_{ij\beta}\nonumber \\
&& \quad -\frac{1}{\sigma_{1}^{3}(B)}\sum_{\alpha,\beta}\sum_{i\in
B} F^{\alpha\beta} (W_{ii\alpha}\sigma_{1}(B)-W_{ii}\sum_{j\in
B}v_{jj\alpha}) (W_{ii\beta}\sigma_{1}(B)-W_{ii}\sum_{j\in
B}v_{jj\beta})\nonumber \\
&& \quad  -\sum_{i\in B}[\sigma_l(G)+\frac{\sigma_{1}^{2}(B|i)-
\sigma_{2}(B|i)}{\sigma_{1}^{2}(B)}]
[\sum_{\alpha,\beta,\gamma,\eta\in
G}F^{\alpha\beta,\gamma\eta}(\Lambda)W_{i\alpha\beta}W_{i\gamma\eta}
+\sum_{\alpha}F^{X^{\alpha}}X^{\alpha}_{ii}\nonumber
\\
&& \quad + 2\sum_{\alpha\beta\in G}F^{\alpha\beta} \sum_{j\in
G}\frac{1}{\lambda_{j}}W_{ij\alpha}W_{ij\beta}+
2\sum_{\alpha,\beta\in
G}\sum_{\gamma=1}^{n+1}F^{\alpha\beta,X^{\gamma}}W_{i\alpha\beta
}X^{\gamma}_i +\sum_{\gamma, \eta=1}^{n+1}
F^{X^{\gamma},X^{\eta}}X^{\gamma}_iX^{\eta}_i\nonumber
\\
&& \quad +2\sum_{\alpha,\beta\in
G}\sum_{\gamma=1}^{n+1}F^{\alpha\beta,u}W_{i\alpha\beta
}u_i +2\sum_{\gamma=1}^{n+1}F^{u,X^{\gamma}}u_iX^{\gamma}_i +F^u u_{ii}+F^{u,u}u_i^2 ].
\end{eqnarray*}

The term involving $X_{ii}$ and the terms involving $u_i, u_{ii}$ ($i\in B$) can be controlled by $W_{ii}$ and $\nabla W_{ii}$ using the Weingarten formula and Lemma \ref{hessian of u}. We obtain
\begin{eqnarray*}\label{new3-1-nn-f}
&& \sum F^{\alpha\beta}\phi_{\alpha\beta}-\phi_t=O(\phi +\sum_{i,j\in
B}|\nabla W_{ij}|)-\frac{1}{\sigma_{1}(B)} \sum_{\alpha,\beta}
\sum_{i,j\in
B,i\neq j}F^{\alpha\beta}W_{ij\alpha}W_{ij\beta}\nonumber \\
&& \quad -\frac{1}{\sigma_{1}^{3}(B)}\sum_{\alpha,\beta}\sum_{i\in
B} F^{\alpha\beta} (W_{ii\alpha}\sigma_{1}(B)-W_{ii}\sum_{j\in
B}v_{jj\alpha}) (W_{ii\beta}\sigma_{1}(B)-W_{ii}\sum_{j\in
B}v_{jj\beta})\nonumber \\
&& \quad  -\sum_{i\in B}[\sigma_l(G)+\frac{\sigma_{1}^{2}(B|i)-
\sigma_{2}(B|i)}{\sigma_{1}^{2}(B)}]
[\sum_{\alpha,\beta,\gamma,\eta\in
G}F^{\alpha\beta,\gamma\eta}(\Lambda)W_{i\alpha\beta}W_{i\gamma\eta}
\nonumber
\\
&& \quad + 2\sum_{\alpha\beta\in G}F^{\alpha\beta} \sum_{j\in
G}\frac{1}{\lambda_{j}}W_{ij\alpha}W_{ij\beta}+
2\sum_{\alpha,\beta\in
G}\sum_{\gamma=1}^{n+1}F^{\alpha\beta,X^{\gamma}}W_{i\alpha\beta
}X^{\gamma}_i +\sum_{\gamma, \eta=1}^{n+1}
F^{X^{\gamma},X^{\eta}}X^{\gamma}_iX^{\eta}_i].
\end{eqnarray*}
The analysis in the proof of Theorem 3.2 in \cite{BG} can be used to show the right hand side of
above inequality can be controlled by $\phi+ |\nabla
\phi|-C\sum_{i,j\in B}|\nabla W_{ij}|$. Inequality (\ref{flow3}) is verified. The first part of theorem follows from the standard strong maximum principle for parabolic equations.

For the second part of theorem, as in \cite{BG}, we may
approximate $M_0$ by a strictly convex $M^{\epsilon}_0$. By
continuity, there is $\delta>0$ (independent of $\epsilon$), such
that there is a solution $M^{\epsilon}(t)$ to (\ref{flow1}) with
$M^{\epsilon}(0)= M^{\epsilon}_0$ for $t\in [0,\delta]$. We argue
that $M^{\epsilon}(t)$ is strictly convex for $t\in [0,\delta]$.
If not, there is $t_0>0$ so that $M^{\epsilon}(t)$ is strictly
convex for $0\le t<t_0$, but there is one point $x_0$ such that
$(h_{ij}(x_0,t_0))$ is not of full rank. This contradicts
the first part of theorem. Taking $\epsilon \to 0$, we conclude
that $M(t)$ is convex for all $t\in [0,\delta]$. This implies that
the set $t$ where $M(t)$ is convex is open. It is obviously
closed. Therefore, $M(t)$ is convex for all $t\in [0,T]$. Again,
by first part of theorem, $M(t)$ is strictly convex for all
$t\in (0,T]$. \qed

\medskip

In the case of flow (\ref{umcf}), $F(g^{-1}h, X,\nu, u, t)=u\sigma_1(g^{-1}h)-n$, it is clear $F$ satisfies conditions in Theorem \ref{thmw2-flow-1}. 
\begin{proof} {\bf( Proof of Corollary \ref{cor 1}.)}
  The proof follows from our main theorem, the monotonicity of the quermassintegrals (Proposition \ref{prop mono quermass}) and Theorem \ref{thmw2-flow-1}.
\end{proof}


\begin{thebibliography}{99}

\bibitem{Andrews0}
B. Andrews, {\em Monotone quantities and unique limits for
evolving convex hypersurfaces,} Internat. Math. Res. Notices {\bf
1997} (1997), 1001--1031.

\bibitem{BLO} J. Barbosa, J. Lira and V. Oliker, {\em A priori estimates for starshaped compact hypersurfaces with prescribed mth curvature function in space forms}, Nonlinear problems in mathematical physics and related topics, I, 35�52, Int. Math. Ser. (N. Y.), 1, Kluwer/Plenum, New York, 2002.

\bibitem{BG}B. Bian and P. Guan, {\em A microscopic convexity principle for nonlinear partial differential equations}, Invent. Math. 177 (2009), no. 2, 307-335.

\bibitem{brakke} K. Brakke, {\em The motion of a surface by its mean curvature}, Mathematical Notes, 20. Princeton University Press, Princeton, N.J., 1978.

\bibitem{CM}E. Cabezas-Rivas and V.Miquel, {\em Volume preserving mean curvature flow in the hyperbolic space}, Indiana Univ. Math. J. 56 (2007), no. 5, 2061-2086.

\bibitem{gage} M. Gage, and R. Hamilton,
{\em The heat equation shrinking convex plane curves},
J. Differential Geom. 23 (1986), no. 1, 69�96.

\bibitem{G} C. Gerhardt, {\em Flow of nonconvex hypersufaces into spheres.} Journal of
Differential Geometry, 32 (1990) 299-314.

\bibitem{G1} C. Gerhardt, {\em Curvature problems}, Series in Geometry and Topology, 39. International Press, Somerville, MA, 2006.

\bibitem{GL} P. Guan and J. Li,  {\em The quermassintegral inequalities for $k$-convex starshaped domains.} Adv. Math. 221 (2009), no. 5, 1725--1732.

\bibitem{GL1} P. Guan and J. Li, {\em A fully non-linear flow and quermassintegral inequalities}, preprint. 

\bibitem{GW1}P. Guan and G. Wang, {\em A fully nonlinear conformal flow on locally conformally
flat manifolds}, Journal fur die reine und angewandte Mathematik, {\bf 557} (2003), 219-238.

\bibitem{GW2} P. Guan and G. Wang, {\em Geometric inequalities on locally conformally flat manifolds}, Duke Math. J. 124 (2004), no. 1, 177-212.
 
\bibitem{H} G. Huisken, {\em Flow by mean curvature of convex surfaces into spheres}, J. Diff. Geom. 20 (1984), 237-266
    
\bibitem{H3} G. Huisken, {\em Contracting convex hypersurfaces in Riemannian manifolds by their mean curvature,} Invent. Math. 84 (1986), no. 3, 463-480.
    
\bibitem{H4} G. Huisken,  {\em The volume preserving mean curvature flow,} J. Reine Angew. Math. 382 (1987), 35-48.
    
\bibitem{H2} G. Huisken and C. Sinestrari, {\em Convexity estimates for mean curvature flow and singularities of mean convex surfaces}, Acta Math. 183 (1999), no. 1, 45�70.

\bibitem{JL} Q. Jin and Y.Y. Li,  {\em Starshaped compact hypersurfaces with prescribed k-th mean curvature in hyperbolic space}, Discrete Contin. Dyn. Syst. 15 (2006), no. 2, 367�377.

\bibitem{Lady} O.A. Ladyzenskaya, V.A. Solonnikov, and N. N. Ural'ceva, {\em Linear and quasi-linear equations of parabolic type}, Translations of Mathematical Monographs, 23, Providence, RI: American Mathematical Society, (1968)

\bibitem{O} V. Oliker, {\em Existence and uniqueness of convex hypersurfaces with prescribed Gaussian curvature in spaces of constant curvature}, in the series of Seminari dell' Istituto di Mathematica Applicata "Giovanni Sansone", Universita degli Studi di Firenze, 1983, 1-39


\bibitem{P} P. Petersen,  {\em Warped product,} ``http://www.math.ucla.edu/~petersen/'', expository work.

\bibitem{R1} R. Reilly, {\em On the Hessian of a function and the curvatures of its graph}, Michigan Math. J. 20 (1973), 373�383.

\bibitem{R2} R. Reilly, {\em Applications of the Hessian operator in a Riemannian manifold}, Indiana Univ. Math. J. 26 (1977), no. 3, 459�472.

\bibitem{Sch} F. Schulze, {\em Nonlinear evolution by mean curvature and isoperimetric inequalities}, J. Differential Geom. 79 (2008), no. 2, 197--241.

\bibitem{topping} P. Topping, {\em Mean curvature flow and geometric inequalities}, J. Reine Angew. Math. 503 (1998), 47-61.

\bibitem{U} J. Urbas, {\em On the expansion of starshaped hypersurfaces
by symmetrics functions of their principal curvatures}. Mathmatische
Zeitschrift, 205,355-372 (1990).
\end{thebibliography}
\end{document}